\documentclass[12pt,a4paper]{article}

\usepackage{amssymb,mathtools}
\usepackage{amsthm}
\usepackage{dsfont}
\usepackage{pgfplots}
\pgfplotsset{compat=newest,cycle list name=exotic}
\usepgfplotslibrary{groupplots}
\usepackage{tikz}
\usepackage{dsfont}
\usepackage{algorithm}
\usepackage[noEnd=true]{algpseudocodex}
\usepackage{enumitem}
\usepackage{xcolor}
\usepackage{xfrac}
\usepackage{booktabs}

\newcommand{\R}{\mathds{R}}
\newcommand{\co}{\operatorname{co}}

\newcommand{\supp}[2]{\operatorname{supp}(#2,#1)}

\newcommand{\dom}{\mathop{\mathrm{dom}}}

\algrenewcommand\alglinenumber[1]{\footnotesize Step #1:}%
\newcommand{\epi}{\mathop{\mathrm{epi}}}

\newtheorem{lemma}{Lemma}
\newtheorem{theorem}{Theorem}

\newtheorem{proposition}{Proposition}
\theoremstyle{definition}
\newtheorem{assumption}{Assumption}
\newtheorem{definition}{Definition}

\newtheorem{example}{Example}
\newtheorem{remark}{Remark}

\usepackage{pgfplots}
\usepackage[hidelinks]{hyperref}
\usepackage{caption}
\usepackage{subcaption}
\usepackage{enumitem}

\newcommand{\Argmin}{\mathop{\mathrm{Argmin}}}

\newcommand{\keywords}[1]{\textbf{Keywords:}#1}
\newcommand{\subclass}[1]{\textbf{MSC Classification:}#1}

\begin{document}

\title{Global minimisation of nonconvex functions by generalising the mirror descent method}
\author{ R. D\'iaz Mill\'an \and J. Ugon}

\maketitle


\begin{abstract}
In this paper, we introduce a mirror descent algorithm for minimising abstract convex functions. This algorithm relies on solving a proximal-type subproblem with an abstract Bregman divergence based proximal term. We also introduce a conceptual proximal point algorithm for abstract convex minimisation. We prove the convergence of both algorithms when the set of abstract linear functions forms a linear space. This latter assumption can be relaxed only to require the set of abstract linear functions to be closed under the sum, which is a classical assumption in abstract convexity.  We provide numerical examples of the minimisation of nonconvex functions with the presented algorithms.

\keywords{ abstract convexity, proximal point, mirror descent, Bregman divergence}

\subclass{49C26 \and 49J52 }
\end{abstract}

\section{Introduction}
Abstract convexity was first proposed by Kutateladze and Rubinov~\cite{kutateladze.ea:1972,kutateladze.ea:1976} to generalize the notion of (Minkowski) duality to nonconvex functions and was extensively developed in the subsequent decades~\cite{andramonov:2002,bui.ea:2021,burachik.ea:2005,burachik.ea:2007,diaz-millan.ea:2023,pallaschke.ea:1997,rubinov:2000,singer:1997}, to name a few. A function is abstract convex with respect to a set of functions $L$ if it is the supremum of functions from $L$.

Many theoretical results from convex analysis have been generalised to the framework of abstract convexity, including duality~\cite{kutateladze.ea:1972}, conjugacy~\cite{rubinov:2000}, the subdifferential~\cite{diaz-millan.ea:2025, rubinov:2000} and normal cones~\cite{jeyakumar.ea:2007}. However, few algorithms have been proposed to build on that theory to minimise abstract convex functions. A notable exception is the cutting angle method~\cite{Andramonov.ea:1999,andramonov:2002} to minimise Lipschitz continuous functions. Some conceptual algorithms are also described in~\cite[chapter 9]{rubinov:2000}. Finally, a proximal point algorithm targeted at minimising weakly convex functions is proposed in~\cite{bednarczuk.ea:2025}. To our knowledge, there exists no algorithmic framework generically applicable in abstract convexity.

Consider the problem of minimising a $f:X\rightarrow \R$ stated as follows: 
 \begin{equation}\label{problem}
\text{minimise } f(x) \text{ subject to } x\in X.
 \end{equation}

Optimisation algorithms are central to tackling optimisation problems like~\eqref{problem}. Within the literature, there is a plethora of algorithms and their variations specifically tailored for convex functions. However, the existence of algorithms for non-convex optimisation problems is limited due to computational costs and the challenge of proving convergence. Indeed, most of these algorithms focus on generalising their convex versions to find stationary points, which satisfy a necessary, but not sufficient condition for (global) optimality. Our focus, in this paper, is different: we aim to generalise an algorithm, namely the mirror descent method, to globally minimise any function, provided it is expressed as an abstract convex one using an appropriate selection of a set of \emph{abstract linear functions}, with the condition that this set forms a linear vector space. We also describe a conceptual Bregman proximal point algorithm for abstract convex functions. 

The first algorithm, referred to as \emph{abstract mirror descent}, exhibits the convergence of the $limitinf$ of the sequence of function values. We call the second algorithm \emph{abstract Bregman proximal-point method}, and we provide proof of convergence for the function values of the iterates. Additionally, we demonstrate that all cluster points in the sequence generated by the algorithm belong to the solution set of the problem. We also provide numerical results to test these algorithms on real (nonconvex) abstract convex problems.

The paper is structured into four sections. Initially, in Section~\ref{sec:preliminaries}, we revisit crucial results and definitions pertaining to abstract convexity, and introduce the notations we use in the paper. In Section~\ref{sec:projection}, we introduce the notion of abstract Bregman divergences, distances and projection. Then, Section~\ref{sec:mirror} and Section~\ref{sec:proximal} are dedicated to presenting the algorithms with their convergence and numerical examples, respectively.

\section{Preliminaries}\label{sec:preliminaries}
In this paper, we are interested in solving the following optimisation problem: 
\begin{equation}\label{eq:prob}
    \min_{x\in X} f(x),
\end{equation}
where $X$ is a vector space.
We start by recalling the main definitions and properties of abstract convexity. We will use the notations introduced by~Rubinov~\cite{rubinov:2000}.
  Let $X$ be a set and $\mathbb{F}$ be the set of all real-valued functions with domain on $X$ as well as the function uniformly equal to $-\infty$:
  \[\mathbb{F} := \{f:X\to \R\cup\{+\infty\}\}\cup \{-\infty:X\to \R, -\infty(x)=-\infty\}.\] 
  Let $L\subset \mathbb{F}$ be a subset of functions and define the abstract convexity function as follows:

  \begin{definition}[Abstract convex function~\cite{rubinov:2000}]
    A function $f$ is said to be \emph{$L$-convex} if there exists a subset $U\subseteq L$ such that for any $x\in X$, $f(x) = \sup_{u\in U} u(x)$.
  \end{definition}

  
  To characterise the minimal $L$-convex set containing a given subset, we introduce the following definition of the abstract convex hull. 
  \begin{definition}[Abstract Convex Hull of a set~\cite{rubinov:2000}]\label{def:convexhullset}
    The intersection of all $L$-convex sets containing a set $C\subset L$ is called the \emph{abstract convex hull} of $C$ and denoted $\co_L C$. 
  \end{definition}
The following property is a direct consequence of the definition.

  \begin{proposition}[{\cite[Proposition 1.1 and Corollary 1.1]{rubinov:2000}}]\label{prop:convexhullsup}
    The $L$-convex hull of a set $C$ is $L$-convex. More specifically, it is the smallest $L$-convex set containing $C$, and $\co_L C = \supp{L}{\sup_{l\in C} l}$.
  \end{proposition}
  
Building on the classical subdifferential from convex analysis, we now define the abstract subdifferential within the framework of abstract convexity.  

  \begin{definition}[Abstract subdifferential~\cite{rubinov:2000}]
    The $L$-subdifferential of a function at a point $x\in X$ is the set
    \[ \partial_L f(x) = \{l\in L: f(y)\geq f(x) + l(y) - l(x), \forall y\in X\}. \]
  \end{definition}

  We also define the abstract conjugate as follows:
  \begin{definition}[$L$-conjugate]
  Let $f$ be an $L$-convex function. Then
  \[f_L^*(u) = \sup_{x\in X} u(x) - f(x) \]
  \end{definition}

  The $L$-conjugate generalises the well known result:
  \begin{proposition}[{\cite[Proposition 7.7]{rubinov:2000}}]\label{prop:moreau}
  Let $f$ be a $L$-convex function on $X$. Then,
      For any $u\in L$ and any $x\in X$, we have that
      \[f(x) + f_L^*(u) \geq u(x)\] with equality if and only if $u\in \partial_L f(x)$.
  \end{proposition}
 
  We recall that the sublevel sets (denoted by $S_c(f)$) and the epigraph of a function $f$ ($\epi(f)$) are defined as follows:
  \begin{align*}
    S_c(f) &:= \{x\in X: f(x)\leq c\}\\
    \epi(f) &:= \{(x,c) \in X\times \R: f(x)\leq c\}.
  \end{align*}
  We say that the sum rule applies if the conditions for the following theorem are satisfied.
  
    \begin{theorem}[Sum Rule~\cite{diaz-millan.ea:2025}]\label{thm:sumsubdif}
  Consider two families of functions with domain on $X$, $L_1$ and $L_2$, and two functions $f_1:X\to \R$ and $f_2:X\to \R$ be $L_1$-convex and $L_2$-convex respectively such that $\dom f_1 \cap \dom f_2 \neq \emptyset$.  Assume that $\supp{H_{L_1}}{f_1}+\supp{H_{L_2}}{f_2} = \supp{H_{L_1}+H_{L_2}}{f_1+f_2}$. Then for all $x\in \dom f_1 \cap \dom f_2$, we have the following:
   \[ \partial_{L_1+L_2} (f_1+f_2) = \partial_{L_1} f_1 (x) + \partial_{L_2} f_2(x).\]
  \end{theorem}
  
\section{Abstract Bregman Divergence}\label{sec:projection}

The Bregman divergence is a robust tool for developing algorithms in convex analysis. Its notion plays a crucial role in simplifying the projection step in many algorithms. In this section, we introduce the concept of Abstract Bregman divergence, which generalises this concept to the classic convex case.
  \begin{definition}[Abstract Bregman Divergence~{\cite{grasmair:2010}}]\label{def:bregman}
    Let $\phi: X \rightarrow \R$ be a $L$-convex function and $x,y\in X$, and let $\lambda\in \partial_L\phi(y)$. We define the $L$-Bregman divergence of $x$ from $y$ with respect to $\lambda$ as 
    \[ D_{\phi}^\lambda(x,y) = \phi(x) - \phi(y) - (\lambda(x) - \lambda(y)) \]
  \end{definition}

    \begin{lemma}\label{lem:triangle}
        Let $a,b,c$ be points in $X$, $\alpha\in \partial_L \phi(a)$ and $\beta\in \partial_L \phi(b)$. Then we have the following equality:
        \begin{equation}\label{eq:triangle}
        D^\alpha_\phi(c,a) + D^\beta_\phi(a,b) - D^\beta_\phi(c,b) = \beta(c) - \beta(a) - (\alpha(c) - \alpha(a))
        \end{equation}
    \end{lemma}
    \begin{proof}
    Using the definition of the abstract Bregman divergence, we have:
    \begin{align*}
    D^\alpha_\phi(c,a) + D^\beta_\phi(a,b) - D^\beta_\phi(c,b) = & \phi(c) - \phi(a) + \alpha(a) - \alpha(c) \\&+\phi(a) - \phi(b) + \beta(b) - \beta(a)\\& -\phi(c) + \phi(b) +\beta(c) - \beta(b)\\
    =& \alpha(a) - \alpha(c) -\beta(a) + \beta(c)
        \end{align*}
        which proves the result.
    \end{proof}

\section{Abstract Mirror Descent Algorithm}\label{sec:mirror}

In this section, we propose a generalised Bregman Mirror Descent algorithm. The Mirror Descent algorithm for the convex optimisation problem was originally proposed by~Nemirovsky and Yudin~\cite{nemirovsky.ea:1983}. 

\begin{algorithm}[H]
  \caption{Bregman Mirror Descent algorithm for abstract convex functions}\label{alg:conceptualmirror}
  \begin{algorithmic}[1]
    \State Pick a $(L,X)$-convex function $\phi$.
    \State Select $x_0\in X$, $\lambda_0\in \partial_L\phi(x_0)$, and $\{c_k\}_{k\geq 0}\in (0,1)$. Set $k=0$.
    \Repeat
      \State Take $u_k\in \partial_L f(x_k)$
      \State Pick $x_{k+1} \in \Argmin_{x\in X}\{ u_k(x) + \frac{1}{c_k} D^{\lambda_k}_\phi(x,x_k)\}$
      \State\label{step:new_subgradient} Define $\lambda_{k+1} = \lambda_k - c_k u_k$
      \State set $k=k+1$.
    \Until{}
  \end{algorithmic}
\end{algorithm}

Before we proceed with proving the convergence of Algorithm~\ref{alg:conceptualmirror}, a few comments are in order.

Notice that we don't need to assume the sum rule for this algorithm (in particular, Step~\ref{step:new_subgradient}) to be well-defined. Furthermore, it can be seen that $\lambda_{k+1}\in \partial_L \phi(x_{k+1})$. Indeed, by the minimality of $x_{k+1}$, for any $y\in X$:
\begin{align*}
c_ku_k(y) + D_\phi^{\lambda_k}(y,x_k) &\geq c_ku_k(x_{k+1}) + D_\phi^{\lambda_k}(x_{k+1}, x_k)\\
c_ku_k(y) - \lambda_{k}(y) + \phi(y) &\geq c_ku_k(x_{k+1}) - \lambda_{k}(x_{k+1}) + \phi(x_{k+1})\\
\phi(y) - \lambda_{k+1}(y) &\geq \phi(x_{k+1}) - \lambda_{k+1}(x_{k+1}) 
\end{align*}

To prove convergence, we will need the function $\phi$ to satisfy the following assumption:
\begin{assumption}\label{ass:boundedsubgradients}
    There exists a $L-$convex function $\phi: X\to \R$, and real positive numbers $M>0,\alpha\geq 0$, such that for any $x,y\in X$, $c>0$ and any $u\in \partial_L f(y)$, we have
    \begin{equation}\label{eq:boundedsubgradient}
      c(u(x)-u(y))\leq D^\mu_\phi(x,y) + c^{\alpha}M,
    \end{equation}
 where $u\in \partial_L \phi(y)$.
\end{assumption}

\begin{remark}
  It is not necessary for $X$ to be a vector space. In fact, we do not use any assumption on the set $X$, other than it being nonempty.
\end{remark}
 
\begin{example}
  In the classical convex case, Assumption~\ref{ass:boundedsubgradients} is verified for any strongly convex function $\phi$, with coefficient $\sigma>0$ whenever $\|u\|_*^2$ is bounded for any $u\in \partial f(x)$ and any $x\in X$ (which is the usual assumption for showing convergence of the Mirror Descent method). Indeed, in this case,
  \[\langle cu, x-y\rangle \leq \frac{\sigma}{2}\|x-y\|^2 + \frac{1}{2\sigma}\|cu\|^*_2\leq  D_\phi^{\lambda}(x,y) + c^2M\]
  where the first inequality is Fenchel-Young inequality applied to $\frac{\sigma}{2}\|x\|^2$ and the second inequality derives from the boundedness of the subgradients and the fact that the $\sigma$-strong convexity of $\phi$ is equivalent to $D_\phi^\lambda(x,y)\geq \frac{\sigma}{2}\|x-y\|^2$ for any $x,y\in X$, $\lambda\in \partial \phi(y)$.
\end{example}

\begin{remark}
  Choosing a function $\phi$ satisfying Assumption~\ref{ass:boundedsubgradients} may be challenging in some cases.
  When the function $f$ is written as the maximum of a finite number of abstract linear functions, however, this assumption only needs to be verified for a finite number of abstract linear functions $u$.

  Another case which is easier to handle is when the set $X$, instead of being a vector space, is a compact set, and the functions $\phi$, $f$, and all functions $u\in L$ are continuous, which is the case usually considered for the mirror descent method in the classical convex case.
\end{remark}


We also require the following assumption:

\begin{assumption}\label{ass:mirror} 
$L$ is a linear space, 
$\sum_{k=1}^{\infty}c_k=\infty$ and $\sum_{k=1}^{\infty}c_k^\alpha<\infty$ with $\alpha>1$ fixed by Assumption~\ref{ass:boundedsubgradients}.
\end{assumption}

Now we will prove, under the Assumptions~\ref{ass:mirror} and~\ref{ass:boundedsubgradients}, that Algorithm~\ref{alg:conceptualmirror} converges to a solution of the Problem~\ref{problem}.

\begin{theorem}\label{thm:mirrorconvergence} Suppose that a solution $x^*$ to Problem~\ref{problem} exists. Let $(x_k)_{k\in \mathbb{N}}$ be the sequence generated by the Algorithm \ref{alg:conceptualmirror}.  Under the Assumptions~\ref{ass:boundedsubgradients} and~\ref{ass:mirror}, we have
  \[\liminf_{l\to \infty} f(x_l) = f(x^*)\]
\end{theorem}

\begin{proof}
Let $x^*$ be a solution to Problem~\eqref{problem} and let $u_k\in \partial_Lf(x_k)$. Then we have that:
\begin{align*}
0 \leq{} & c_k(f(x_k)- f(x^*))\nonumber\\
\leq{} & c_k (u_k(x_k) - u_k(x^*))\nonumber\\
={} & c_k(u_k(x_k) - u_k(x_{k+1}) + c_k(u_k(x_{k+1}) - u_k(x^*))
\end{align*}

We have from Step~\ref{step:new_subgradient} and Lemma~\ref{lem:triangle}.
\begin{align*}
c_k(u_k(x_{k+1}) - u_k(x^*)) =& \lambda_k(x_{k+1}) - \lambda_{k+1}(x_{k}) - \lambda_k(x^*) + \lambda_{k+1}(x^*)\\
=& D_\phi^{\lambda_k}(x^*,x_k) -  D_\phi^{\lambda_{k+1}}(x^*,x_{k+1}) - D_\phi^{\lambda_k}(x_{k+1},x_k) 
\end{align*}

Using Assumption~\ref{ass:boundedsubgradients}, there exist $\alpha>=0$ and $M>0$ such that:
\[
c_k(u_k(x_k) - u_k(x_{k+1}) \leq c_k^\alpha M + D_\phi^{\lambda_k}(x_{k+1},x_k)
\]

From this we can conclude that 
  \begin{equation}\label{eq:lemf}
  c_k (f(x_k) - f(x^*))\leq D_\phi^{\lambda_k}(x^*,x_k)- D_\phi^{\lambda_{k+1}}(x^*,x_{k+1})+Mc_k^\alpha
  \end{equation}

  Summing for $k=0$ to $l$ in equation \eqref{eq:lemf}, we find that
  \[ \sum_{k=0}^l c_k f(x_{k}) - s_lf(x)\leq D_\phi^{\lambda_{0}}(x,x_{0})- D_\phi^{\lambda_{l+1}}(x,x_{l+1})  + \sum_{k=0}^l M c_k^a\alpha ,  \]
 where $s_l=\sum_{k=0}^l c_k$. 
  Since $D^\lambda_\phi(y,z)\geq 0$ for any $y, z$, we can simplify the inequality above to:
  \[  s_l\left( \min_{1\leq k\leq l}f(x_{k}) - f(x)\right)\leq  D_\phi^{\lambda_{0}}(x,x_{0}) + \sum_{k=0}^l M c_k^\alpha    \]
  That is, for any $x\in X$,
  \[  \min_{1\leq k\leq l+1}f(x_k) \leq f(x) +\frac{D_\phi^{\lambda_{0}}(x,x_{0})}{s_l }+\frac{M\sum_{k=0}^l c_k^\alpha}{s_l}. \]
  Since $\lim_{l\to \infty} s_l = +\infty$ by Assumption~\ref{ass:mirror},  then we obtain that for $x = x^*$,
  \[  \liminf_{l\to \infty} f(x_l) \leq f(x^*) \]

  From this, we conclude that 
    \[  \liminf_{l\to \infty} f(x^*) \leq f(x^*)\qedhere \]
\end{proof}

\begin{remark}
  Note that by construction, $x_{k+1}$ is chosen to be the projection of $x_k$ onto an abstract hyperplane. Indeed, for any $x\in X$ we have that $c_k u_{k}(x)  + D_{\phi}^{\lambda_k}(x,x_k) \geq c_k u_{k}(x_{k+1})  + D_{\phi}^{\lambda_k}(x_{k+1},x_k)$.

  In particular, when $u_k(x)\leq u_k(x_{k+1})$, this implies that $D_{\phi}^{\lambda_k}(x,x_k)\geq D_\phi^{\lambda_k}(x_{k+1},x_k)$.
\end{remark}

\subsection{Implications for the classical convex case}

Assumption~\ref{ass:boundedsubgradients} is more general than the classical assumption that $\phi$ is strongly convex and that the subgradients of the function $f$ are bounded, which allows us to generalise the mirror descent method even in the classical convex case. Indeed, assume that for any $x,y\in X$, $D_\phi(x,y) \geq \sfrac{\sigma}{p}\|x-y\|^p$, for some $\sigma>0$ and $p>1$, and then set $q := \frac{p}{p-1}$. We have the following Fenchel-Young inequality for any $c>0$ and any $u\in X^*$:
\[c\langle u,x-y\rangle\leq \sfrac{\sigma}{p}\|x-y\|_p^p + \sfrac{c^q}{\sigma q}\|u\|^q_q.\]

If we assume that $\|u\|_q^q$ is bounded for any $u\in \partial f(x)$ and any $x\in X$, we can conclude that Assumption~\ref{ass:boundedsubgradients} is verified. In other words, assuming that there exists $\sigma>0$ such that $\phi$ follows the inequality
\[\phi(y)\geq \phi(x) + \langle u, y-x\rangle + \sfrac{\sigma}{p}\|x-y\|_p^p\] for any $x,y\in X$, $u\in \partial f(x)$, it is possible to apply the mirror descent method with step sizes $c_k$ such that
$\sum_{k=1}^{\infty}c_k = +\infty$ and $\sum_{k=1}^{\infty}c_k^{\sfrac{p}{p-1}} < \infty$

\section{Numerical Experiments}

In this section, we demonstrate how Algorithm~\ref{alg:conceptualmirror} can be applied to some examples of abstract convex functions.


Let $X=\R$, $L=\{x\mapsto ax^3+bx\}$ and consider the function $f\coloneqq \max(p_1, p_2)$, where
\begin{align*}
  p_1(x) &= x^3-12x\\
  p_2(x) &= -x^3+6x
\end{align*}

A plot of the function $f$ is shown in Figure~\ref{fig:maxcubic}. Its unique global minimiser is found at $x=3$, where $f(x) = -9$. The function $f$ has two other local minimisers, at $x=-3$ and $x=0$.

\begin{figure}[ht]
\begin{center}
\begin{tikzpicture}
  \begin{axis}[axis lines=center]
    \addplot gnuplot[no markers,samples=1000,very thick] { ((x**3-12*x > -x**3+6*x) ? x**3-12*x : -x**3+6*x) with lines};
  \end{axis}
\end{tikzpicture}
\caption{Plot of the test function for the first experiment}\label{fig:maxcubic}
\end{center}
\end{figure}

We apply Algorithm~\ref{alg:conceptualmirror} with $\phi = \frac{3}{4}x^4$ and set $\lambda_0(x) = \langle \nabla \phi(x_0),x\rangle = \langle x_0^3, x\rangle$.
At every iteration of the algorithm, we need to minimise the function $u_k(x) + D^{\lambda_k}_\phi(x,x_k)$. This function is a polynomial of degree 4 such that the coefficient of the largest degree is always $\sfrac{3}{4}$. Therefore, it always has a minimum. We solve it analytically by first finding all the stationary points of this polynomial (which requires finding the real roots of a cubic polynomial) and then identifying the minimiser(s).
We run the algorithm for 5000 iterations, from the starting points $x_0\in \{-5,-0.5,0.5,5\}$ and using different sequences $c_n$. The results are summarised in Table~\ref{tab:maxcubic}. We provide the value of the function at the 100-th iterate, as well as the iteration $k^*$ at which $f(x_k)-f(x^*) < 10^{-2}$. Figure~\ref{fig:resultsmaxcubic} shows the function values over 5000 iterations for different initial points and different sequences $c_k$.

\begin{table}[!ht]
\begin{center}
\begin{tabular}{cccr}\toprule
$x_0$ & $c_n$                & $x_{100}$ & $k^*$\\\midrule
-5    & $\sfrac{10}{n}$ & $3.01562$ & $118$\\
-0.5  & $\sfrac{10}{n}$ & $3.00914$ & $7$\\
0.5   & $\sfrac{10}{n}$ & $3.01167$  & $82$\\
5     & $\sfrac{10}{n}$ & $2.99686$ & $90$\\\bottomrule
-5    & $\sfrac{1}{n}$ & $-3.0012$ & $>5000$\\
-0.5  & $\sfrac{1}{n}$ & $1.17016$ & $3672$\\
0.5   & $\sfrac{1}{n}$ & $2.48154$  & $188$\\
5     & $\sfrac{1}{n}$ & $3.00124$ & $21$\\\bottomrule
-5    & $\sfrac{1}{n^{0.8}}$ & $0.74113$ & $4363$\\
-0.5  & $\sfrac{1}{n^{0.8}}$ & $2.43474$ & $133$\\
0.5   & $\sfrac{1}{n^{0.8}}$ & $2.99377$  & $126$\\
5     & $\sfrac{1}{n^{0.8}}$ & $3.00165$ & $45$\\\bottomrule
\end{tabular}
\caption{Results of the abstract mirror descent method for various initial points and step sizes}\label{tab:maxcubic}
\end{center}
\end{table}

\begin{figure}[!ht]
\begin{center}
\begin{tikzpicture}
  \begin{groupplot}[group style={group size=2 by 1, xlabels at=edge bottom, ylabels at=edge left},small,xlabel=$n$,ylabel=$f(x_n)-f^*$,width=0.5\columnwidth]
  \nextgroupplot[axis lines=center,ymax=20, title={$c_n=\sfrac{1}{n}$},legend to name={mirroriteratevalues}, legend style={legend columns=2},ylabel near ticks, xlabel near ticks]
    \addplot+[mark size=1pt] table [x=k, y=x1] {results/maxcubic_1_1.dat};
    \addplot+[mark size=1pt] table [x=k, y=x2] {results/maxcubic_1_1.dat};
    \addplot+[mark size=1pt] table [x=k, y=x3] {results/maxcubic_1_1.dat};
    \addplot+[mark size=1pt] table [x=k, y=x4] {results/maxcubic_1_1.dat};
    \legend{$x_0=-5$, $x_0=-0.5$, $x_0=0.5$, $x_0=5$};
  \nextgroupplot[axis lines=center,ymax=20, title={$c_n=\sfrac{1}{n^{0.8}}$}, xlabel near ticks]
    \addplot+[mark size=1pt] table [x=k, y=x1] {results/maxcubic_1_0.8.dat};
    \addplot+[mark size=1pt] table [x=k, y=x2] {results/maxcubic_1_0.8.dat};
    \addplot+[mark size=1pt] table [x=k, y=x3] {results/maxcubic_1_0.8.dat};
    \addplot+[mark size=1pt] table [x=k, y=x4] {results/maxcubic_1_0.8.dat};
  \end{groupplot}
  \path (group c1r1.south east) -- node[below=20pt]{\ref{mirroriteratevalues}} (group c2r1.south west);
\end{tikzpicture}
\caption{Function values at the iterates for the Mirror descent method.}\label{fig:resultsmaxcubic}
\end{center}
\end{figure}

\begin{figure}[!ht]
\begin{center}
\begin{tikzpicture}
  \begin{axis}[axis lines=center]
    \addplot gnuplot[no markers,samples=1000] { ((x**3-12*x > -x**3+6*x) ? x**3-12*x : -x**3+6*x) with lines};
    \addplot+[mark=+,only marks,mark size=1.5pt] table [x=x1, y=f1] {results/iterates_maxcubic.dat};
  \end{axis}
\end{tikzpicture}
\caption{Iterates for $x_0=-5$ and $c_n=\sfrac{2}{n^{0.8}}$.}\label{fig:maxcubiciterates}
\end{center}
\end{figure}

The results of these experiments confirm that the abstract mirror descent method presented in Algorithm~\ref{alg:conceptualmirror} indeed globally minimises abstract convex functions. The method achieves this convergence from a variety of initial points, and with different step size choices $c_k$. Interestingly, unlike many global methods, it does not ``escape'' or ``jump'' over local minimisers, but mostly ``ignores'' them, as shown in Figure~\ref{fig:maxcubiciterates}.

Results from Table~\ref{tab:maxcubic} show that the choice of the step size $c_k$ matters in practice. Even though in all experiments except one the method reached the solution in less than the maximum number of iterations we set (5000), there was considerable variation in the time required to do so, varying from 7 to 3672 iterations, starting from the same initial solution ($x_0=-0.5$). This is a known limitation of the mirror descent method in the convex case, and is therefore expected behaviour here.

\section{Proximal point method}\label{sec:proximal}

In this section, we introduce a generalisation of the proximal point method to the abstract convex case, with Bregman distance. The (convex) Bregman proximal minimisation algorithms were originally proposed by~Censor and Zenios~\cite{censor.ea:1992} (see also \cite{kiwiel:1997}).



Consider $\phi, f: X\rightarrow \R$ be two $L$-convex functions.

\begin{algorithm}[H]
  \caption{Bregman proximal point algorithm for abstract convex functions}\label{alg:proximal}
  \begin{algorithmic}[1]
    \State Select $x_0\in X$, $\lambda_0\in \partial_L \phi(x_0)$ and $\{c_k\}_{k\geq 0} \subset \R^+$. Set $k=0$.
    \Repeat
      \State Pick $x_{k+1} \in \Argmin_{x\in X} \{f(x) +  \frac{1}{c_k}D^{\lambda_k}_\phi(x,x_k)\}$
      \State Set $\lambda_{k+1} = \lambda_k - c_kg_{k+1}$, where $g_{k+1}\in \partial_L f(x_{k+1})$.\label{step:subgradientchoice}
      \State Set $k=k+1$.
      \Until{$0\in \partial_L f(x_k)$}
  \end{algorithmic}
\end{algorithm}

\begin{assumption}\label{ass:proximal} 
  \begin{enumerate}[label=(A\arabic*)]
    \item\label{ass:bounded_sublevel} The sublevel sets of the function $D^\lambda_\phi(\cdot,y)$ are bounded, for any $\lambda \in \partial_L \phi(y)$ and $y\in X$.

\item \label{ass:ck} One of the following two conditions is satisfied:
\begin{enumerate}
    \item $L-L\subset L$ and $c_k=1$ for any $k\geq 1$
    \item $L$ is a linear space,  $\sum_{k=1}^{\infty}c_k=\infty$.
\end{enumerate}

  \end{enumerate}
\end{assumption}

\begin{lemma} When Assumption~\ref{ass:proximal}.\ref{ass:ck} is satisfied,
Step~\ref{step:subgradientchoice} is well defined, provided that the sum rule for the subdifferential (Theorem~\ref{thm:sumsubdif}) is satisfied. 
\end{lemma}

\begin{proof}
    Indeed, in this case a necessary and sufficient condition for optimality is that $0\in \partial_{{\bar{L}}} \phi_{k}(x_{k+1})$, where $\phi_{k}(x) = f(x) +  \frac{1}{c_k}D^{\lambda_k}_\phi(x,x_k)$ and $\bar{L}=L+L-\sfrac{\lambda_k}{c_k}$. Since the sum rule applies, this means that there exists $g_{k+1}\in \partial_L f(x_{k+1})$ and $\lambda_{k+1}\in \partial_L \phi(x_{k+1})$ such that $0 = g_{k+1} + \frac{1}{c_k}(\lambda_{k+1} - \lambda_k)$. This implies that $\lambda_{k+1}$ is well defined.
\end{proof}

For a discussion of the sum rule, we refer the reader to~\cite{diaz-millan.ea:2025}.
 
\begin{theorem}\label{thm:proxmethodconvergence}

Suppose that Assumption~\ref{ass:proximal} is satisfied and the sum rule applies. Let $(x_k)_{k\in \mathbb{N}}$ the sequence of iterates produced by Algorithm~\ref{alg:proximal}, then

  \[\lim_{l\to \infty} f(x_l) = \inf_{x\in X} f(x).\]
\end{theorem}

\begin{proof}

  Applying Lemma~\ref{lem:triangle}, we find that at every iteration $k$ and for any $x\in X$,
  \begin{multline*}
    D_\phi^{\lambda_{k+1}}(x,x_{k+1}) + D_\phi^{\lambda_{k}}(x_{k+1},x_{k}) - D_\phi^{\lambda_{k}}(x,x_{k}) \\= \lambda_{k}(x) - \lambda_{k}(x_{k+1}) - (\lambda_{k+1}(x) - \lambda_{k+1}(x_{k+1}))\\
    =  c_k(g_{k+1}(x) - g_{k+1}(x_{k+1}))\leq c_k (f(x) - f(x_{k+1})).
  \end{multline*}

  Summing this inequality over $0\leq k\leq l$ and defining  $s_k = \sum_{j=0}^k c_j$, gives
  \[ D_\phi^{\lambda_{l+1}}(x,x_{l+1}) - D_\phi^{\lambda_{0}}(x,x_{0}) + \sum_{k=0}^l D_\phi^{\lambda_{k}}(x_{k+1},x_{k}) \leq s_l f(x) - \sum_{k=0}^l c_kf(x_{k+1}),\] which implies that
   \[ D_\phi^{\lambda_{l+1}}(x,x_{l+1}) - D_\phi^{\lambda_{0}}(x,x_{0})  \leq s_l f(x) - \sum_{k=0}^l c_k f(x_{k+1}),\] 

  Since the sequence $(f(x_k))_{k\in \mathcal{N}}$ is decreasing, because of the definition of the $x_{k+1}$, and since $D^\lambda_\phi(y,z)\geq 0$ for any $y, z \in X$, we can simplify the inequality above to:
  \[  - D_\phi^{\lambda_{0}}(x,x_{0})  \leq s_l (f(x) - f(x_{l+1}))  \]
  That is, for any $x\in X$,
  \[  f(x_l) \leq f(x) + \frac{D_\phi^{\lambda_{0}}(x,x_{0})}{s_{l-1}}.  \]
  Since $\lim_{l\to \infty} s_l = +\infty$ by the assumption, and since Assumption~\ref{ass:proximal}~\ref{ass:bounded_sublevel} ensures that the function $D_\phi^{\lambda_{0}}(\cdot,x_{0})$ has a bounded sublevel set $S_\alpha = \{x: D^{\lambda_0}(x,x_0)\leq \alpha < +\infty\}$, then we obtain that for any $x\in S_\alpha$,
  \[  \lim_{l\to \infty} f(x_l) \leq f(x) \]

  From this, we conclude that 
  \[  \lim_{l\to \infty} f(x_l) = \inf_{x\in X} f(x)\]
\end{proof}

%
%

\begin{theorem}
  If the set $X^*$ is not empty, and the function $f:X\rightarrow \R$ is lower semi-continuous, then every cluster point of the sequence $x_k$ belongs to $X^*$.
\end{theorem}
  \begin{proof}

    Let $\bar{x}$ be a cluster point of the sequence $x_k$. By lower semi-continuity of the function $f$, this implies that $f(\bar{x}) \leq \lim_{k\to \infty} f(x_k)= \inf_{x\in X} f(x)$ and therefore $f(\bar{x})=\inf_{x\in X} f(x)$.
  \end{proof}

  \begin{remark}
    By definition of $x_{k+1}$, we have:
    \begin{equation}\label{eq:xkmin}
      f(x_{k+1}) + c_k D_\phi(x_{k+1},x_k) \leq f(x) + c_k D_\phi(x,x_k).
    \end{equation}
    Therefore, if $f(x)\leq f(x_{k+1})$, then $D^{\lambda_k}_\phi(x_{k+1},x_k)\leq D^{\lambda_k}_\phi(x,x_k)$, and so the point $x_{k+1}$ is a $D^{\lambda_k}_\phi$-projection of $x_k$ onto the $L$-convex set $\Omega = \{x\in X: f(x)\leq f(x_{k+1})\}$.
  \end{remark}

The proximal point method, even the context of convex minimisation, is mostly conceptual and is generally not applied in practice, as it requires solving a subproblem that is as hard as the original problem. It is, however, theoretically important.

\section{Conclusion}

The main purpose of this paper is to propose conceptual algorithms based on the framework of abstract convexity.

We developed two such algorithms, one an abstract mirror descent algorithm and the other an abstract proximal point algorithm, both using a Bregman-like proximal term. We showed their convergence, under classical assumptions, and demonstrated how they can be applied to minimise abstract convex functions. For the abstract mirror descent method, we provided numerical experiments that demonstrate how it can be used to globally minimise nonconvex functions with multiple local minimisers.

One restriction of any proximal-based algorithm (including the convex programming ones) is that its convergence depends on the sum rule for subdifferentiability. Previous work, by the authors of this paper and by other authors, provides conditions for this property to be satisfied. In future research, we will investigate families of abstract linear functions satisfying these properties and adapt our algorithms. The other restriction of proximal-based algorithms is the need to solve a subproblem, and our future research will focus on solving these subproblems for specific classes of functions.

\section*{Acknowledgements}

This research was funded through the ARC discovery project DP180100602.

\section*{Conflict of interest}

The authors declare that they have no conflict of interest.
 
\bibliographystyle{plain}      
\bibliography{references.bib}

\end{document}